\documentclass[11pt,reqno]{amsart}
\usepackage[T1]{fontenc}
\usepackage[utf8]{inputenc}
\usepackage{amsmath,amssymb,amsthm,mathtools}
\usepackage{enumitem}
\usepackage{booktabs}
\usepackage{array}
\usepackage{algorithm}
\usepackage{algpseudocode}
\usepackage[margin=1.15in]{geometry}
\usepackage{xcolor}
\usepackage[colorlinks=true,linkcolor=blue!50!black,citecolor=blue!50!black]{hyperref}

\theoremstyle{plain}
\newtheorem{theorem}{Theorem}[section]
\newtheorem{proposition}[theorem]{Proposition}
\newtheorem{lemma}[theorem]{Lemma}
\newtheorem{corollary}[theorem]{Corollary}

\theoremstyle{definition}
\newtheorem{definition}[theorem]{Definition}
\newtheorem{remark}[theorem]{Remark}
\newtheorem{example}[theorem]{Example}

\newcommand{\C}{\mathbb{C}}
\newcommand{\Z}{\mathbb{Z}}
\newcommand{\N}{\mathbb{N}}
\newcommand{\Q}{\mathbb{Q}}

\newcommand{\Gm}{\mathbb{G}_m}
\newcommand{\Ga}{\mathbb{G}_a}
\newcommand{\GL}{\operatorname{GL}}
\newcommand{\Gal}{\operatorname{Gal}}
\newcommand{\dGal}{\delta\text{-}\Gal}
\newcommand{\Aut}{\operatorname{Aut}}
\newcommand{\Sol}{\operatorname{Sol}}
\newcommand{\Tor}{\operatorname{Tor}}
\newcommand{\Hom}{\operatorname{Hom}}
\newcommand{\Ann}{\operatorname{Ann}}
\newcommand{\ord}{\operatorname{ord}}

\newcommand{\Seq}{\operatorname{Seq}}
\newcommand{\trdeg}{\operatorname{trdeg}}
\newcommand{\dtrdeg}{\delta\text{-}\trdeg}
\newcommand{\ddim}{\delta\text{-}\dim}
\newcommand{\A}{\mathcal{A}}
\newcommand{\B}{\mathcal{B}}

\newcommand{\Lc}{\mathcal{L}}
\newcommand{\Dc}{\mathcal{D}}
\newcommand{\V}{\mathbb{V}}
\newcommand{\Newt}{\operatorname{New}}

\begin{document}

\title[Linear cellular automata and difference Galois theory II]
{Generalized linear cellular automata in groups\\ and difference Galois theory II}

\author{David Bl\'azquez-Sanz}
\address{Universidad Nacional de Colombia -- Sede Medell\'in,
Carrera 65 No.\ 59A-110, Medell\'in, Antioquia, Colombia}
\email{dblazquezsa@unal.edu.co}

\author{Mario Alejandro Vergara Tapiero}
\address{Universidad Nacional de Colombia -- Sede Medell\'in,
Carrera 65 No.\ 59A-110, Medell\'in, Antioquia, Colombia}
\email{mavergarat@unal.edu.co}


\begin{abstract}
In the first part of this series the $\sigma$-ring spanned by the periodic solutions of a
generalized linear cellular automaton in a discrete group was shown to be Hopf--Galois,
with pro-algebraic Galois group, while a Galois theory for the $\sigma$-ring spanned by the
finite support solutions was left open, with parameterized difference Galois theory
suggested as the tool. We apply that tool to automata over $\Z$. The Fourier
transform turns the automaton into a rank one equation
$\sigma(\hat x)=\hat\alpha\hat x$ over a field on which $\delta=z\,d/dz$ acts, and we
compute, and algorithmically decide, its parameterized Galois group: for genuine automata
with symbols rational in the time variable it is either $\Gm(\C)$ or $\Gm$, the dividing
line being separability of the symbol, $\hat\alpha(z,t)=\gamma(z)\psi(t)$. Finally, the
parameterized group in the time direction decides holonomy in the space direction; as a
corollary, the Stirling numbers of the first kind are not spatially holonomic.
\end{abstract}

\maketitle
\setcounter{tocdepth}{1}
\tableofcontents

\section{Introduction}\label{sec:intro}

A generalized linear cellular automaton in a discrete group $G$ is a system of linear
difference equations with one variable per element of $G$, invariant under the right
translation action, and with a transition rule of finite support. As explained in
\cite{BM}, such a system is a convolution equation
\begin{equation}\label{eq:conv}
\sigma(s)=a\star s,\qquad a\in C_c(G,k),
\end{equation}
where $k$ is a $\sigma$-field of sequences. 

If $G$ is finite, then Picard--Vessiot--Franke theory applies directly. If moreover $G$ is abelian, the Fourier transform that linearises convolution
is a gauge transformation that diagonalises the system and the Galois group
is realized inside a linear torus \cite[Thm.~12]{BM}. For infinite $G$
this is no longer true, and we should consider three different $\sigma$-rings of solutions: the ring $L$ spanned by all solutions, the ring $L_c$ spanned by the solutions of finite support, and the ring $L_{per}$ spanned by the periodic solutions. The periodic case was solved in
\cite{BM}: $L_{per}$ is the direct limit of the Picard--Vessiot extensions of the
finite rank automata induced on the finite quotients $G/H$, the extension $k\subseteq
L_{per}$ is Hopf--Galois, and its Galois group is the projective limit of the algebraic
Galois groups, hence pro-algebraic \cite[Thm.~15]{BM}.
 
The finite support case can not be analysed in the framework of Picard--Vessiot--Franke theory.
Remark 5 of \cite{BM} states the problem and points at
the tool: \emph{``In the abelian case, the ring of Fourier polynomials is endowed of a
differential structure that allows to apply parameterized difference Galois theory.''} Here we
study in depth the case $G=\Z$. We do not construct a Galois theory for $L_c$ itself; what
we compute is the parameterized Galois group of the Fourier transformed equation
$\sigma(\hat x)=\hat\alpha\hat x$ over the $\sigma\delta$-field $K=\C(z)^{\rm diff}(t)$,
where $\delta=z\,d/dz$ acts on the frequency variable. We give the computation of the
$\delta$-Galois group of a rank one equation in general (Theorem \ref{thm:main}), which for
coefficients rational in $t$ leads to a trichotomy (Theorem \ref{thm:tricho}): the
$\delta$-Galois group is either finite, $\Gm(\C)$ or $\Gm$. For genuine automata the finite
case does not occur, and the reduction of the group to the constants is equivalent to the
separability of the symbol $\hat\alpha$ (Theorem \ref{thm:sep}).

Finally we show that the $\delta$-Galois in the time direction $t$ decides holonomy in the
space direction $n$. We say that an automaton is \emph{spatially holonomic} if the
coefficients $x_n(t)$ of its fundamental solution satisfy one linear recurrence in $n$ with
coefficients in $\C[n,t]$, valid for all $n$ and all $t$. This condition turns out again to
be equivalent to separability of the symbol, and gives the recurrence explicitly when it
holds (Theorem \ref{thm:holoequiv}). As an illustration, the array of unsigned Stirling
numbers of the first kind is not spatially holonomic, hence not P-finite in two variables
(Corollary \ref{cor:stirling}). This statement is treated as known in the symbolic
summation literature \cite{Kauers07,CKS09}; we have not found a published proof, although,
as we point out in Remark \ref{rmk:egf}, a short direct argument is available for this
particular array through its exponential generating function. The point of the Galois
theoretic approach, whose closest methodological precedent is \cite{BDR}, is that it gives
a criterion valid for the whole family of automata rather than an argument case by case.

\section{Generalized linear cellular automata in groups}\label{sec:auto}

Here we recall the setting of \cite[\S1]{BM}, in the special case of $G=\Z$.

\subsection{Automata and convolution equations}

Let $\Seq(\C)$ be the ring of complex sequences with componentwise operations and shift
$\sigma(s)(t)=s(t+1)$, and let
\[
\Seq_\infty(\C):=\Seq(\C)/\!\sim,\qquad
s\sim r\ \Longleftrightarrow\ s(t)=r(t)\ \text{for all large }t,
\]
be the ring of \emph{germs at infinity} of sequences. As any rational function has finitely many poles,
\[
\C(t)\hookrightarrow\Seq_\infty(\C),\qquad f\longmapsto\text{germ of }\big(f(t)\big)_{t},
\]
is a well defined embedding of $\sigma$-rings. The same applies to $\kappa(t)\subseteq
\Seq_\infty(\kappa)$ for any $\sigma$-field $\kappa$. We fix a $\sigma$-subfield
$\kappa\subseteq\Seq_\infty(\C)$ with $\kappa^\sigma=\C$. All identities between sequences
below are identities of germs, that is, they are asserted for all large $t$.

\begin{definition}[\cite{BM}]
A \emph{generalized linear cellular automaton} in $\Z$ with coefficients in $\kappa$ and radius $r$ is the system
\begin{equation}\label{eq:auto}
s_{n}(t+1)=\sum_{h= -r}^{r}a_{h}(t)\,s_{n-h}(t),\qquad n\in\Z,\ a_h\in \kappa,
\end{equation}
equivalently the convolution equation $\sigma(s)=a\star s$ with $a\in C_c(\Z,\kappa)$ supported in $\{-r, \ldots, r\}$.
\end{definition}

Following \cite[\S1.4]{BM} we write $\Sol(a,R)$, $\Sol_c(a,R)$ and $\Sol_{T\text{-}per}(a,R)$
for the spaces of all solutions, of finite support solutions, and of $T$-periodic
solutions with values in a $\sigma$-ring $R$, and
\[
L=\kappa[s_n],\qquad L_c=\kappa[s_n:\ s\in\Sol_c],\qquad
L_{per}=\bigcup_{T\geq1}L_{T\text{-}per}
\]
for the $\sigma$-rings they span. The \emph{fundamental solution} $f$ is the one with
$f_{n}(t_0)=\delta_{n,0}$, where the base time $t_0$ is fixed in \eqref{eq:t0} below, and
the fundamental $T$-periodic solution is $f_T=f\star \sum_{k\in\Z} e_{kT}$ where $e_{n}$ stands for the mass function that maps $n$ to $1$ and any other element of $\Z$ to $0$.

\subsection{Fourier transform}

\subsubsection*{The circle}
The dual of $\Z$ is the circle $S^1$; a character is $\rho_\theta(n)=e^{in\theta}$. We
write $z=e^{i\theta}$ throughout, so that $S^1$ is seen as the unit circle in the
$z$-plane and the torsion subgroup of $\widehat\Z$ is the group $\mu_\infty$ of roots of
unity, identified with $\Q/\Z$ via $\zeta_q=e^{2\pi iq}$. For $f\in C_c(\Z,R)$ the Fourier
transform is
\[
\hat f=\sum_{n\in\Z}f_n\,z^{n}\ \in\ R[z^{\pm1}],
\]
a Laurent polynomial; for arbitrary $f\in C(\Z,R)$ it is a formal Fourier series, and the
inversion formula is $f_n=\frac1{2\pi}\int_0^{2\pi}e^{-in\theta}\hat f\,d\theta$.
The transform exchanges convolution and product and commutes with $\sigma$, so
\eqref{eq:auto} becomes the rank one equation
\begin{equation}\label{eq:rank1}
\sigma(\hat x)=\hat\alpha\,\hat x,\qquad
\hat\alpha =\sum_{h}a_h\,z^{h}\in \kappa[z^{\pm1}].
\end{equation}
We call $\hat\alpha$ the \emph{symbol} of the automaton and write
$$\mu=\ord_z\hat\alpha = \min\{h\colon a_h\neq 0\} \leq\nu=\deg_z\hat\alpha = \max\{h\in \mathbb Z\colon a_h \neq 0\},$$ so that the Newton polytope of $\hat\alpha$, as a Laurent polynomial in $z$, is the segment $\Newt(\hat\alpha)=[\mu,\nu]$.

\subsubsection*{Cyclic groups}
The dual of $\Z/T\Z$ is the group $\mu_T$ of $T$-th roots of unity; with
$\xi=e^{2\pi i/T}$ the transform is the Vandermonde matrix $(\xi^{nj})$ and
$\hat a_j=\sum_{n}\xi^{nj}a_n$. The convolution equation in $\Z/T\Z$ is thereby
decoupled into $T$ scalar equations $\sigma(\hat x_j)=\hat a_j\hat x_j$, solved by
$\kappa$-hypergeometric sequences \cite[Ex.~2]{BM}. Compatibility of the two transforms is the
commutative diagram of \cite[proof of Thm.~9]{BM}: the transform of the quotient automaton
on $\Z/T\Z$ is the restriction of $\hat\alpha$ to $\mu_T\subset S^1$.

\subsection{The group algebra}

Let
\[
\A=\C[\Z]=\C[z^{\pm1}]=\mathcal O(\Gm)
\]
be the group algebra of $\Z$. It is a principal ideal domain and a unique factorization domain; its units are exactly the
monomials $cz^{a}$. The two structures we use are the quotients
$\A/(z^T-1)\cong\prod_{\zeta^T=1}\C$, which are the coordinate rings of the torsion
subgroups $\mu_T$, and the fraction field $\C(z)$, the local ring of the generic point of
$\Gm$.

Since the coefficients of $\hat\alpha$ are germs, once representatives are chosen
(for $\kappa=\C(t)$, the values of the rational functions at integers),
$\hat\alpha(z,s)\in\C[z^{\pm1}]$ is defined for all large $s$, and it is nonzero for all large $s$ because a nonzero germ is
invertible. We therefore fix once and for all an integer
\begin{equation}\label{eq:t0}
t_0\in\N\ \text{ such that }\ \hat\alpha(z,s)\ \text{ is defined and nonzero for every }
s\geq t_0,
\end{equation}
and we base the fundamental solution at $t_0$. Its Fourier transform is
\[
\beta(z,t) =\prod_{s=t_0}^{t-1}\hat\alpha(z,s)\ \in\ \C[z^{\pm1}],\qquad t\geq t_0 .
\]
Note that $\beta(z,t)$ is a sequence of nonzero Laurent polynomials, and not necessarily
an element of $\kappa[z^{\pm1}]$. For $t>t_0$ its support is contained in
$\big[(t-t_0)\mu,(t-t_0)\nu\big]$, since degrees are additive in the domain
$\C[z^{\pm1}]$. If moreover the extreme coefficients satisfy $a_\mu(s)\,a_\nu(s)\neq0$ for
every $s\geq t_0$, both endpoints are attained; in general $a_\mu$ and $a_\nu$, being
nonzero germs, vanish at finitely many $s$ only, so that
\[
\ord_z\beta(z,t)-(t-t_0)\mu
\qquad\text{and}\qquad
(t-t_0)\nu-\deg_z\beta(z,t)
\]
are nonnegative and bounded independently of $t$; this asymptotic form is all we use.
Replacing $t_0$ by a larger $t_0'$ \emph{divides} $\beta$ by the Laurent polynomial
$\prod_{s=t_0}^{t_0'-1}\hat\alpha(z,s)\in\A$, which affects nothing in
\S\ref{sec:sd}, \S\ref{sec:algo} and \S\ref{sec:holo}. When $\hat\alpha(z,s)$ is defined and nonzero
for \emph{every} $s\in\N$ one has $t_0=0$, the usual normalization; we indicate the value
of $t_0$ whenever it is not $0$.

\begin{theorem}[\protect{\cite[Thm.~8]{BM}}]\label{thm:BM8}
The Fourier transform restricts to an isomorphism of $\A$-modules
\[
\Sol_c(a,L_c)\ \xrightarrow{\ \sim\ }\ \beta\cdot\A .
\]
That is, the space of finite support solutions is free of rank one over the group algebra,
generated by the fundamental solution.
\end{theorem}

\section{$\sigma$-Galois theory}\label{sec:sigma}

\subsection{Picard--Vessiot--Franke theory}

We refer to \cite{SvdP} and \cite{HSS} for the general Picard--Vessiot-Franke theory, also known as linear difference Galois theory. Here, we only recall some relevant facts. 
Let $(F,\sigma)$ be a $\sigma$-field of characteristic zero with algebraically closed field of $\sigma$-constants $C=F^\sigma$. A $\sigma$-ring extension $F\subseteq R$ is a \emph{Picard--Vessiot
ring} for $\sigma(X)=AX$, $A\in\GL_d(F)$, if $R$ is $\sigma$-simple, generated by the
entries of a fundamental matrix and the inverse of its determinant. It exists and is unique
up to isomorphism. A constructive way of looking at the Picard--Vessiot ring is to consider the $\sigma$-ring $F[\GL_d] = F[U,1/\det U]$, where $U$ is a matrix of variables and the operator $\sigma$ is defined by $\sigma(U) = AU$. Let $\mathcal I$ be a maximal $\sigma$-ideal of $F[\GL_d]$ and $X$ the class of $U$ modulo $\mathcal I$, so that $F[\GL_d]/\mathcal I =  F[X,1/\det X]$ is a Picard--Vessiot ring. Note that $R = F[\V(\mathcal I)]$ so that the Picard-Vessiot ring is the ring of regular functions on the $F$-variety $\V(\mathcal I)$ which is called a \emph{torsor} for the difference equation. The group,
$\Gal(\V(\mathcal I)) =\{B \in \GL_d(C) \colon \V(\mathcal I) B = \V(\mathcal I)\}$
turns out to be an algebraic group defined over $C$ such that, 
\begin{enumerate}
\item $\Gal(\V(\mathcal I)) \xrightarrow{\sim} \Aut_\sigma(R/F)$ by means of $B\mapsto \phi_B$ where $\phi_B(X) = XB$.
\item The map of $F$-varieties,
$$\V(\mathcal I) \times \Gal(\V(\mathcal I))_F \to \V(\mathcal I), \quad (X , B) \mapsto XB$$
gives to $\V(\mathcal I)$ an structure of principal homogeneous space modeled over $\Gal(\V(\mathcal I))_F$.
\end{enumerate}
There is a Galois correspondence with the intermediate $\sigma$-rings of the total ring of
fractions and algebraic subgroups of $\Aut_\sigma(R/F)$ \cite[Thm.~1.29]{SvdP}.

We will need the explicit computation of the Galois group of order one and diagonal systems of difference equations. 

\begin{definition}
The multiplicative group of $\sigma$-coboundaries of $F$ is
$$\B_\sigma(F)=\{\sigma(g)/g:g\in F^\times\}.$$    
\end{definition}

\begin{proposition}[rank one, \protect{\cite[\S2.2]{SvdP}}]\label{prop:rank1}
Let $\lambda\in F^\times$. The Galois
group of the difference equation 
$\sigma(x)=\lambda x$
over $F$ is:
\begin{enumerate}
  \item The cyclic group of $m$-th roots of $1$, $\mu_m\subset\Gm(C)$, if the class $[\lambda]$ in $F^\times/\B_\sigma(F)$ is an element of finite order $m$, that is, $m$ is the smallest natural number such that $\lambda^m$ is a $\sigma$-coboundary.
  \item The multiplicative group $\Gm(C)$ if $[\lambda]\in F^\times/\B_\sigma(F)$ has no finite order, that is, no power of $\lambda$ is a $\sigma$-coboundary.
\end{enumerate}
\end{proposition}

\begin{proposition}[\protect{\cite[\S2.2]{SvdP}}]\label{prop:diag}
Let $c_i\in F^\times$, $i\in J$ finite, and let $R$ be the Picard--Vessiot ring of the
diagonal system $\sigma(x_i)=c_ix_i$. Then
\[
\Aut_\sigma(R/F)=\Big\{\lambda\in\Gm(C)^J:\ \prod_i\lambda_i^{m_i}=1
\ \ \forall m\in\mathfrak m_J\Big\},\qquad
\mathfrak m_J=\Big\{m\in\Z^J:\ \prod_i c_i^{m_i}\in\B_\sigma(F)\Big\}.
\]
\end{proposition}

\subsection{Results of \cite{BM}}

Three results of \cite{BM} are used below. For finite $G$ the Fourier transform is a gauge
transformation and reduces the structure group \cite[Thms.~11 and 12]{BM}, so that the
Galois groups of the automaton and of its transform coincide and
$\dim_\C\Aut_\sigma(L/\kappa)\leq|G|$. For discrete abelian $G$,
\begin{equation}\label{eq:BM9}
L_{per}=\kappa\Big[\beta_\zeta:\ \zeta\in\Tor(\widehat G)\Big]
\end{equation}
\cite[Thm.~9]{BM}, where $\beta_\zeta(t)=\beta(\zeta,t)$ denotes the value of the
transformed fundamental solution at the torsion character $\zeta$. Finally, $\kappa\subseteq L_{per}$ is a Hopf--Galois extension whose Galois
group
\[
\Gal(L_{per}/\kappa):=\varprojlim_{H}\Aut_\sigma(L_{H\text{-}per}/\kappa)
\]
is pro-algebraic and coincides with $\Aut_\sigma(L_{per}/\kappa)$
\cite[Lem.~13, Prop.~14, Thm.~15]{BM}. In the abelian case, \cite[Ex.~9]{BM} describes it
inside the infinite torus $\prod_{\zeta}\C^\times$ by the equations $\lambda_\zeta=1$ at
the zeros of $\hat\alpha$ together with the multiplicative resonances. We record that
description in closed form.

\subsection{The periodic group is a character group}

\begin{definition}\label{def:Lambda}
Let $\hat\alpha\in \kappa[z^{\pm1}]$ and set
\[
\Lambda\ :=\ \big\langle\,[\hat\alpha(\zeta)]\ :\ \zeta\in\mu_\infty,\
\hat\alpha(\zeta)\neq0\,\big\rangle\ \leq\ \kappa^\times/\B_\sigma(\kappa),
\]
the \emph{resonance group} of the automaton: the subgroup generated by the classes of the
values of the symbol at the torsion points of $\widehat\Z$.
\end{definition}

\begin{theorem}\label{thm:dual}
There is a canonical isomorphism of pro-algebraic groups
\[
\Gal\big(L_{per}/\kappa\big)\ \cong\ \Hom\big(\Lambda,\C^\times\big).
\]
Consequently the identity component is a pro-torus of rank $\operatorname{rk}\Lambda$, the
component group is the character group of $\Tor(\Lambda)$, and $\Gal(L_{per}/\kappa)$ is finite
if and only if $\Lambda$ is finite.
\end{theorem}

\begin{proof}
Let $I=\{\zeta\in\mu_\infty:\hat\alpha(\zeta)\neq0\}$. By \eqref{eq:BM9}, $L_{per}$ is
generated over $\kappa$ by the solutions of the diagonal family
$\sigma(\hat x_\zeta)=\hat\alpha(\zeta)\hat x_\zeta$, $\zeta\in I$. For $J\subseteq I$
finite let $L_J$ be the corresponding Picard--Vessiot ring, that is
$\kappa\big[\beta_\zeta^{\pm1}:\zeta\in J\big]$; the inverses have to be adjoined for
$\sigma$-simplicity, as in \cite[Ex.~9]{BM}, and accordingly the Galois group of the
periodic solutions is computed through these rings,
$\Gal(L_{per}/\kappa)=\varprojlim_J\Aut_\sigma(L_J/\kappa)$, which is the projective limit of
\cite[Prop.~14]{BM} indexed by the finite sets of torsion characters. By Proposition \ref{prop:diag},
$\Aut_\sigma(L_J/\kappa)$ is cut out in $(\C^\times)^J$ by the module $\mathfrak m_J$, which is
precisely the kernel of the morphism $\Z^J\to \kappa^\times/\B_\sigma(\kappa)$ sending $e_\zeta$ to
$[\hat\alpha(\zeta)]$. Denoting by $\Lambda_J$ its image,
\[
\Aut_\sigma(L_J/\kappa)=\Hom(\Z^J/\mathfrak m_J,\C^\times)=\Hom(\Lambda_J,\C^\times).
\]
Now $\Lambda=\varinjlim_J\Lambda_J$, and $\Hom(-,\C^\times)$ takes direct limits to
inverse limits, which gives the isomorphism. Each
$\Hom(\Lambda_J,\C^\times)$ is algebraic since $\Lambda_J$ is finitely generated, so that we obtain the
pro-algebraic structure.
\end{proof}

\begin{remark}
Theorem \ref{thm:dual} is \cite[Ex.~9]{BM} in closed form. Its usefulness is that it
identifies the periodic Galois group with a concrete arithmetic object: the multiplicative
group generated by the values of the symbol at roots of unity, modulo coboundaries.
\end{remark}

\section{$\sigma\delta$-Galois theory}\label{sec:sd}

\subsection{Differential algebraic groups}

Let $(k,\delta)$ be a differentially closed field of characteristic zero with field of
constants $C=k^\delta$, which is then algebraically closed. A $k$-$\delta$-group is a group
defined by differential polynomial equations over $k$; we use the invariants
$\ddim_k(G)$ (the differential dimension), $\ord(G)=\trdeg_kk\langle G\rangle$ and
$\dim\overline G^{\,\mathrm{Zar}}$, which in general differ. We write $k[\delta]$ for the
Ore ring of linear differential operators $\sum_{j\leq r}c_j\delta^j$, $c_j\in k$; it is a
principal left ideal domain.

The following theorem from P. Cassidy \cite{Cassidy} (see also
\cite[Thm.~3.68]{HSS}) gives us an important tool. Any Zariski dense
subgroup of the multiplicative group is given by a linear differential equation in the logarithmic derivative. 

\begin{proposition}[Cassidy]\label{prop:cassidy}
The $k$-$\delta$-subgroups of $\Ga$ are exactly the $\V(L)=\{y:L(y)=0\}$ with
$L\in k[\delta]$, the case $L=0$ giving $\Ga$ itself. Via the logarithmic derivative
$\ell:\Gm\to\Ga$, defined $\ell(y) = \delta(y)/y$, whose kernel is $\Gm(C)=C^\times$, the $k$-$\delta$-subgroups of $\Gm$
that are Zariski dense are of the form $\V(L(\ell y)) = \{y:L(\ell(y))=0\} $.
\end{proposition}

\subsection{Parameterized Picard--Vessiot theory}\label{ss:parameterized}

Let $(K,\sigma,\delta)$ be a $\sigma\delta$-field of characteristic zero with $\sigma\delta = \delta\sigma$, such that
$k=K^\sigma$ is a $\delta$-closed field, and let $C=k^\delta$. Given a system of linear difference equations $\sigma(X)=AX$ with $A\in\GL_d(K)$, a parameterized Picard--Vessiot ring is a
$\sigma\delta$-simple $\sigma\delta$-ring $R$ generated over $K$ by a fundamental matrix and
the inverse of its determinant. The existence and uniqueness of the parameterized Picard--Vessiot ring under the above hypotheses is established in \cite{HS,Wibmer}. Its group of $\sigma\delta$-automorphisms $\dGal(R/K)$ is a
$k$-$\delta$-subgroup of $\GL_d(k)$ and there is a Galois correspondence
\cite[\S4]{HSS}. 

As above, a constructive way of dealing with parameterized Picard--Vessiot rings is to consider the ring $K\{U,1/\det U\} = K[U, 1/\det U, \delta U, \delta^2 U, \ldots]$ of regular $\delta$-functions on $\GL_d(K)$, where $\sigma$ is defined by $\sigma U = AU$, and a maximal $\sigma\delta$-ideal $\mathcal I\subset K\{U,1/\det U\}$, so that $R = K\{U,1/\det U\}/\mathcal I$. Now, we have a $\delta$-torsor $\V(\mathcal I)$ which is a Kolchin closed $\delta$-variety and a principal homogeneous space over $\dGal(R/K)$. The relation between the parameterized and classical Picard--Vessiot theory is clear:
\begin{enumerate}
    \item The Zariski closure of the $\delta$-torsor $\V(\mathcal I)$ is a torsor for the same equation in the classical Picard--Vessiot theory.
    \item The Zariski closure of $\dGal(R/K)$ is the Galois group for the same equation in the classical Picard--Vessiot theory \cite[Prop.~4.38]{HSS}. 
\end{enumerate}

The invariant we use is
\begin{equation}\label{eq:dimtrdeg}
\dtrdeg\big(Q(R)\mid K\big)=\ddim_k\big(\dGal(R/K)\big),
\end{equation}
where $Q(R)$ denotes the total ring of fractions of $R$; see \cite[Prop.~4.36]{HSS}. 



\subsection{$\delta$-Galois group of rank one equations.} Here we give the computation of rank one equations under the general hypothesis of \S\ref{ss:parameterized}. 
\begin{equation}\label{eq:rank1_general}
\sigma(x)=\alpha\,x,\qquad \alpha\in K^\times, \qquad \alpha^{m}\notin\B_\sigma(K)\quad\text{for all }m\geq1 .
\end{equation}
We do not allow any power of $\alpha$ to be a $\sigma$-coboundary in order to avoid the case of finite Galois group. This is equivalent to say that the $\delta$-Galois group is Zariski dense in the multiplicative $\Gm(k)$. In particular the solution $x$ in the Picard-Vessiot extension is $\delta$-transcendental over $K$ if and only if the group is all of $\Gm(k)$.

\begin{definition}
An element $f\in K$ is said to be \emph{$\sigma$-summable}\footnote{This terminology comes from the context of sequences, in which $\sigma$ is the shift operator: $f=\sigma(g)-g$ means $g(t) = g(t_0) + \sum_{s=t_0}^{t-1} f(s)$, so that $g$ is an indefinite sum of $f$.} if it is of the form $f = \sigma(g) - g$ for some $g$ in $K$. 
\end{definition}

The space of $\sigma$-summable elements $(\sigma-1)K$, as $\sigma$ commutes with $\delta$, is a $k[\delta]$-submodule of $K$. For $\alpha\in K$ let us consider the set,
\[
\Lc(\alpha)=\big\{L\in k[\delta]:\ L(\ell(\alpha))\in(\sigma-1)K\big\}
\]
consisting of all differential operators $L$ such that $L(\ell(\alpha))$ is $\sigma$-summable.

\begin{lemma}\label{lem:module}
$\Lc(\alpha)$ is a left ideal of $k[\delta]$. Thus, $\Lc(\alpha)=k[\delta]L_0$ for a unique monic $L_0$, or
$\Lc(\alpha)=0$. Moreover $\Lc(c\alpha)=\Lc(\alpha)$ for $\delta$-constant
$c\in K^\times$, and $\Lc\big(\alpha\,\sigma(h)/h\big)=\Lc(\alpha)$ for all
$h\in K^\times$.
\end{lemma}
\begin{proof}
Note that $K / (\sigma - 1)K$ is a
$k[\delta]$-module and $\Lc(\alpha) =\Ann_{k[\delta]}([\ell(\alpha)]) $ is the annihilator ideal of $[\ell(\alpha)]$. As
$k[\delta]$ is a principal left ideal domain the stated form follows. For the last two
assertions, $\ell(c\alpha)=\ell(c)+\ell(\alpha)=\ell(\alpha)$ when $\delta c=0$, and
$\ell(\alpha\sigma(h)/h)=\ell(\alpha)+(\sigma-1)\ell(h)$.
\end{proof}

\begin{lemma}[fundamental identity]\label{lem:fund}
The logarithmic derivative sends $\sigma$-coboundaries to $\sigma$-summable elements:
$$\ell \colon \B_\sigma(K) \to (\sigma-1)K, \quad \alpha \to \ell(\alpha).$$
Moreover, let $x$ in any $\sigma\delta$-extension of $K$ satisfy $\sigma(x)=\alpha x$. Then
\begin{equation}\label{eq:fundid}
\sigma(\ell(x))- \ell(x)= \ell(\alpha).
\end{equation}
\end{lemma}
\begin{proof}
$\sigma(\ell( x))=\dfrac{\delta(\sigma (x))}{\sigma x}=\dfrac{\delta(\alpha x)}
{\alpha x}=\dfrac{\delta(\alpha)}{\alpha}+\dfrac{\delta(x)}{x}$.
\end{proof}

\begin{theorem}\label{thm:main}
Let $R$ be a parameterized
Picard--Vessiot ring for equation \eqref{eq:rank1_general} over $K$ and let
$G=\dGal(R/K)\subseteq\Gm(k)$.
\begin{enumerate}
\item If $\Lc(\alpha)=0$ then $G=\Gm(k)$, $\ddim_kG=1$ and $\dtrdeg(Q(R)\mid K)=1$, the
solution $x\in R$ is $\delta$-transcendental over $K$.
\item If $\Lc(\alpha)=k[\delta]L_0$ with $L_0$ monic of order $r$, then
\[
G=\V\big(L_0(\ell y)\big)\subseteq\Gm,\qquad \ddim_kG=0,\qquad \ord(G)=r+1,
\]
and, choosing $g\in K$ with $L_0(\ell(\alpha))=\sigma(g)-g$, any solution $x\in R$ satisfies an explicit relation
\begin{equation}\label{eq:relation}
L_0\big(\ell(x)\big)-g\ \in\ k .
\end{equation}
\end{enumerate}
\end{theorem}

\begin{proof}
\emph{Step 1. } Let $M\in\Lc(\alpha)$ and choose
$g\in K$ with $M(\ell(\alpha))=\sigma(g)-g$. Put $w =M(\ell(x))-g$. The coefficients of $M$ lie in
$k=K^\sigma$, so $M$ commutes with $\sigma$, and by Lemma \ref{lem:fund},
\[
\sigma(w)=M(\sigma(\ell x))-\sigma(g)=M(\ell(x) + \ell(\alpha))-\sigma(g)
=M(\ell(x))+M( \ell(\alpha))-\sigma(g)=w ,
\]
so $w\in Q(R)^\sigma=k$. If $\phi\in G$ and $\phi(x)=\lambda  x$ with
$\lambda\in k^\times$, then $\phi$ fixes $w\in k$, while
\[
\phi(w)=M\big(\ell(\lambda  x)\big)-g=M(\ell\lambda)+M(\ell  x)-g=M(\ell\lambda)+w .
\]
Hence $M(\ell\lambda)=0$ for all $\lambda\in G$ and all $M\in\Lc(\alpha)$; in
particular $G\subseteq\V(L_0(\ell(y)))$ and \eqref{eq:relation} holds.

\emph{Step 2.} If $M_1\in k[\delta]$ satisfies $G\subseteq\V(M_1(\ell(y)))$
then $M_1(\ell  x)$ is $G$-invariant, hence lies in $Q(R)^{G}=K$ by the Galois
correspondence \cite[Prop.~4.33]{HSS}; call it $g_1$. Applying $\sigma-1$ and
Lemma \ref{lem:fund} gives $M_1(\ell(\alpha))=\sigma(g_1)-g_1$, that is $M_1\in\Lc(\alpha)$. Thus
\begin{equation}\label{eq:star}
\{M_1\in k[\delta]:\ G\subseteq\V(M_1(\ell(y)))\}=\Lc(\alpha).
\end{equation}

\emph{Step 3.} Since no power of $\alpha$ is a $\sigma$-coboundary, the
classical Galois group of $\sigma(x)=\alpha x$ over $K$ is $\Gm(k)$ by
Proposition \ref{prop:rank1}. Hence $G$ is
Zariski dense in $\Gm(k)$ \cite[Prop.~4.38]{HSS}, so $G\supseteq\ker\ell=C^\times$ by
Proposition \ref{prop:cassidy}, and
therefore $G=\ell^{-1}(\ell(G))$ with $\ell(G)$ a $k$-$\delta$-subgroup of $\Ga$.

\emph{Step 4.} By Proposition \ref{prop:cassidy}, either $\ell(G)=\Ga$ or
$\ell(G)=\V(M_1)$ with $M_1\neq0$. If $\Lc(\alpha)=0$, then by \eqref{eq:star} no
nonzero $M_1$ satisfies $G\subseteq\V(M_1(\ell y))$, so $\ell(G)=\Ga(k)$ and $G=\Gm(k)$; the
statement on $\dtrdeg$ is \eqref{eq:dimtrdeg}. If $\Lc(\alpha)=k[\delta]L_0$, then
$\ell(G)\subseteq\V(L_0)$ by Step 1, so $\ell(G)=\V(M_1)$ with $\V(M_1)\subseteq\V(L_0)$;
by \eqref{eq:star}, $M_1\in\Lc(\alpha)$, that is $M_1=PL_0$ for some $P\in k[\delta]$, whence
$\V(M_1)\supseteq\V(L_0)$. Therefore $\ell(G)=\V(L_0)$ and $G=\V(L_0(\ell y))$.

\emph{Step 5.} Let $u_1,\dots,u_r$ be a $C$-basis of $\V(L_0)$; its Wronskian is
nonzero. A generic point $\lambda$ of $G$ satisfies $\ell(\lambda)=\sum_ic_iu_i$ with
$c_1,\dots,c_r$ independent constants, and the Wronskian matrix expresses the $c_i$ in
terms of $\ell(\lambda),\delta\ell(\lambda),\dots,\delta^{r-1}\ell(\lambda)$. Hence
$k\langle\lambda\rangle=k(\lambda,\delta\lambda,\dots,\delta^{r}\lambda)$ has transcendence
degree $r+1$ over $k$.
\end{proof}

\begin{corollary}[autonomous symbols over the two base fields]\label{cor:auton}
Let us consider equation \eqref{eq:rank1_general} with $\alpha\in k^\times$ not a root of
unity; in both cases below this is exactly the hypothesis that no power of $\alpha$ is a
$\sigma$-coboundary, since $\B_\sigma(K)\cap k^\times=\{1\}$ (for $K=k(t)$ see
Lemma \ref{lem:cob}; for $K=k$ it is obvious, $\sigma$ being the identity on $k$).
\begin{enumerate}
\item \emph{Assume $K = k(t)$}, every $g\in k$ is summable, since $g=(\sigma-1)(t g)$.
Hence $L_0=1$, $r=0$,
\[
G=\Gm(C)=\{\lambda:\delta\lambda=0\},\qquad \ord(G)=1,
\]
and the relation \eqref{eq:relation} reads $\ell(x)=t\,\ell(\alpha) +v$ with $v\in k$.
\item \emph{Assume $K=k$}, here $(\sigma-1)K=0$, so $\Lc(\alpha)=\Ann_{k[\delta]}(\ell(\alpha))$.
If $\delta(\alpha)\neq0$ then $L_0=\delta-\ell(\ell(\alpha))$, $r=1$,
\[
G=\V\big(\ell(\alpha)\,\delta(\ell y)-\delta(\ell(\alpha))\,\ell(y)\big)
=\big\{\lambda\in\Gm(k):\ \ell(\lambda)\in C\cdot \ell(\alpha)\big\},\qquad \ord(G)=2 ,
\]
and, since $L_0(\ell(\alpha))=0$ and $(\sigma-1)K=0$ allow the choice $g=0$, the relation
\eqref{eq:relation} reads
\[
\big(\delta-\ell(\ell(\alpha))\big)\,\ell(x)\in k,
\qquad\text{equivalently}\qquad
\delta\left(\frac{\ell(x)}{\ell(\alpha)}\right)\in k .
\]
If $\delta(\alpha)=0$ then $\ell(\alpha)=0$, $\Lc=k[\delta]$ and $G=\Gm(C)$.
\end{enumerate}
\end{corollary}
\begin{proof}
$\ell(\alpha)\in k$, so $M(\ell(\alpha))\in k$ for every $M\in k[\delta]$. Part (1) is the displayed identity. In
(2), $\Lc(\alpha)$ is the annihilator of $\ell(\alpha)$ in $k[\delta]$, which is
$k[\delta](\delta-\ell(\ell(\alpha)))$ for $\ell(\alpha) \neq0$ and $k[\delta]$ for $\ell(\alpha)=0$; 
and $\V(\delta-\ell(\ell(\alpha)))=C\cdot \ell(\alpha)$, because $k^\delta=C$, whence the description of $G$. The two forms of the relation agree because
$\big(\delta-\ell(\ell(\alpha))\big)u=\ell(\alpha)\,\delta(u/\ell(\alpha))$ and $\ell(\alpha)\in k^\times$.
\end{proof}


\begin{example}[one automaton, two groups]\label{ex:twobases}
Take the symbol $\hat\alpha=z+z^{-1}$ of the Fourier transform of Wolfram's rule 90, for which
\[
b:=\ell(\hat\alpha)=\frac{z^{2}-1}{z^{2}+1},\qquad
\ell(b)=\ell(\ell(\hat\alpha))=\frac{4z^{2}}{z^{4}-1}.
\]
Over $K=\C(z)^{\rm diff}(t)$ one gets $G=\Gm(\C)$ of order $1$ and the relation
$\ell(\hat x)=t\,b$, which under $z=e^{i\theta}$ is
$\partial_\theta\hat x/\hat x=-t\tan\theta$. Over $K= \C(z)^{\rm diff}$ one gets instead
\[
G=\Big\{\lambda:\ \delta(\ell\lambda)=\tfrac{4z^{2}}{z^{4}-1}\,\ell(\lambda)\Big\}
=\big\{\mu\,\hat\alpha^{\,c}:\ \mu\in\C^\times,\ c\in\C\big\},
\]
of order $2$, where $\hat\alpha^{\,c}$ denotes any element of $k$ with
$\ell(\hat\alpha^{\,c})=c\,\ell(\hat\alpha)$, which exists in $k$ because it is $\delta$-closed. 
It is unique up to a factor in $\C^\times$, and it lies in $\C(z)$ only for $c\in\Z$. We have
checked that $\ell(\mu\hat\alpha^{c})=c\,b$ and that these $\lambda$ satisfy the Wronskian
equation. 

\end{example}

\subsection{The trichotomy}

\emph{From here to the end of the paper we focus on the difference equation \eqref{eq:rank1} of the Fourier transform of cellular automata with coefficients in $\kappa = \C(t)$. Thus we apply the general theory developed in \S\ref{ss:parameterized} to the case $k = \C(z)^{\rm diff}$, the differential closure of $(\C(z),\delta)$ with $\delta=z\,d/dz$, and $K = k(t)$, where $\sigma$ acts as the identity on $k$ and by $\sigma(t)=t+1$, and $\delta$ is extended to $K$ by $\delta(t)=0$. Then $\sigma\delta=\delta\sigma$, $K^\sigma=k$ and $k^\delta=\C$.}

As the field $k$ is algebraically closed, for $\hat\alpha \in k(t)^\times$ we have a factorization:
\begin{equation}\label{eq:factor}
\hat\alpha=\gamma\prod_{i=1}^{N}(t-\rho_i)^{e_i},
\qquad \gamma\in k^\times,\ \rho_i\in k\ \text{distinct},\ e_i\in\Z\setminus\{0\},
\end{equation}
and then $\hat\alpha$ is determined, up to the multiplicative constant $\gamma$, by its divisor
$$\operatorname{div}(\hat\alpha) = \sum_{i=1}^N e_i\rho_i.$$
Note that we can also consider the action of $\sigma$ on divisors, by setting 
$$\sigma^{*} \left(\sum_{i=1}^N e_i\rho_i \right) =   \sum_{i=1}^N e_i(\rho_i-1)$$
so that $$\operatorname{div}(\sigma(\hat\alpha)) = \sigma^*(\operatorname{div}(\hat\alpha)).$$

As the number of zeroes and poles of $\hat\alpha$ is finite, each $\Z$-orbit $O\subseteq\{\rho_i\}$ has a first element $\rho_O$ and any other element of $O$ is of the form $\rho_O + m$ for some positive integer $m$. Here and below, an \emph{orbit} means the intersection of a $\Z$-orbit with the finite set
$\{\rho_i\}$. Since $\delta(\rho+m)=\delta\rho$ for $m\in\Z$, all elements of an orbit $O$ have
the same derivative $d_O:=\delta\rho_O$; we call $O$
\emph{moving} if $d_O\neq0$, and we set $s_O=\sum_{\rho_i\in O}e_i$.

\begin{lemma}[coboundaries over $k(t)$]\label{lem:cob}
With the notation of \eqref{eq:factor},
\[
\exists\,m\geq1:\ \hat\alpha^{\,m}\in\B_\sigma(K)
\iff
\big(s_O=0\ \text{for every orbit }O\big)\ \text{and}\ \gamma\ \text{is a root of unity}.
\]
In particular $\B_\sigma(K)\cap k^\times=\{1\}$.
\end{lemma}
\begin{proof}
The divisor in $t$ of $\sigma(g)/g$ is $\sigma^{*}D-D$ with $D=\operatorname{div}(g)$. If
$D$ has multiplicity $f_n$ at $\rho+n$, that of $\sigma^{*}D-D$ is $f_{n+1}-f_n$, whose
sum along each orbit vanishes because $f$ has finite support. Hence
$\hat\alpha^{m}\in\B_\sigma(K)$ forces $m\,s_O=0$, that is $s_O=0$, for every orbit.
Conversely if $s_O=0$ for every orbit, taking $f_n:=\sum_{j<n}e_{\rho+j}$ produces
$g_0$ with $\sigma(g_0)/g_0=\prod_i(t-\rho_i)^{e_i}$, so that
$\hat\alpha^m\in\B_\sigma(K)$ if and only if $\gamma^m\in\B_\sigma(K)\cap k^\times$.
Finally if $\sigma(g)/g=c\in k^\times$ the same divisor computation gives
$g\in k^\times$ and $c=1$.
\end{proof}

\begin{lemma}[summability of rational functions \cite{Abramov}]\label{lem:abramov}
Let $\Omega$ be a field of characteristic zero with $\sigma|_\Omega=\mathrm{id}$ and let $\sigma\colon \Omega(t)\to \Omega(t)$ be the shift operator $\sigma(t)=t+1$. Let 
$f\in\Omega(t)$ have partial fraction decomposition
$$f=p(t)+\sum_{O}\sum_{n}\sum_{j\geq1} \frac{c^{O}_{n,j}}{(t-\rho_O-n)^{j}},$$ where $O$ runs over the
$\Z$-orbits of poles. Then $f\in(\sigma-1)\Omega(t)$ if and only if
$\sum_nc^{O}_{n,j}=0$ for every orbit $O$ and every $j\geq1$.
\end{lemma}
\begin{proof}
$\sigma-1$ is surjective on $\Omega[t]$ in characteristic zero and preserves the
decomposition; on the block $(O,j)$ it acts as ``shift minus identity'' on the finitely
supported sequence $(c_{n,j})_n$, whose image is exactly the set of sequences of zero sum.
\end{proof}

\begin{lemma}\label{lm:dicho1}
Let $\hat\alpha\in K^\times$ be as in \eqref{eq:factor}. Then
\[
\Lc(\hat\alpha)=
\begin{cases}
k[\delta] & \text{if } s_O=0 \text{ for every moving orbit } O,\\[2pt]
0 & \text{otherwise.}
\end{cases}
\]
\end{lemma}

\begin{proof}
From \eqref{eq:factor},
\begin{equation}\label{eq:bexpl}
b=\ell(\hat\alpha)=\ell(\gamma)-\sum_{i=1}^{N}\frac{e_i\,\delta(\rho_i)}{t-\rho_i}.
\end{equation}
Thus $b$ has only simple poles in $t$, with residues $-e_i\delta(\rho_i)$, which vanish at
the $\delta$-constant poles, and $\ell(\gamma)\in k$ has no pole in $t$.

Suppose that $s_O=0$ for every moving orbit. Then $\ell(\gamma)=(\sigma-1)(t\ell\gamma)$ 
is summable, and for a moving orbit $O$ with common value $d_O=\delta\rho\neq0$ the orbital
sum of residues is $-d_O\,s_O=0$, so that part is summable by Lemma \ref{lem:abramov}.
The $\delta$-constant orbits do not contribute. Hence $b\in(\sigma-1)K$ and
$1\in\Lc(\hat\alpha)$.

Suppose now that $s_O\neq0$ for some moving orbit $O$, and let
$A=\sum_{j=0}^{r}c_j\delta^j$ with $c_r\neq0$. From
$\delta\big(c(t-\rho)^{-1}\big)=\delta(c)(t-\rho)^{-1}+c\,\delta(\rho)(t-\rho)^{-2}$ one
gets by induction that the coefficient of $(t-\rho)^{-(j+1)}$ in
$\delta^{j}\big(c(t-\rho)^{-1}\big)$ equals $j!\,c\,\delta(\rho)^{j}$. Applying this to
\eqref{eq:bexpl} with $c=-e_i\delta(\rho_i)$, the coefficient of $(t-\rho_i)^{-(r+1)}$ in
$A(b)$ is $-c_r\,r!\,e_i\,d_O^{\,r+1}$ for $\rho_i\in O$, and nothing else in
\eqref{eq:bexpl} contributes to it. The orbital sum of the coefficients of order $r+1$ is therefore
$-c_r\,r!\,d_O^{\,r+1}\,s_O\ \neq\ 0$.
Since $\delta^j$ raises the order of a simple pole by at most $j$, the order $r+1$ is the
largest one occurring, and Lemma \ref{lem:abramov} gives $A(b)\notin(\sigma-1)K$. As
$A\neq0$ was arbitrary, $\Lc(\hat\alpha)=0$. 
\end{proof}

\begin{theorem}[trichotomy]\label{thm:tricho}
Let $\hat\alpha\in K^\times$ be as in \eqref{eq:factor}. Then the difference equation \eqref{eq:rank1}
falls in one of the three mutually excluding cases:
\begin{enumerate}
\item $s_O = 0$ for every orbit $O$, moving or not, and $\gamma$ is a root of unity, of
exact order $m$. Then $\hat \alpha^m \in \B_{\sigma}(K)$, $\Lc(\hat\alpha) = k[\delta]$, and
the $\delta$-Galois group of \eqref{eq:rank1} is the group $\mu_m$ of $m$-th roots of unity.
\item The above case does not hold, but $s_O = 0$ for every \emph{moving} orbit $O$. Then
$\Lc(\hat\alpha) = k[\delta]$, and the $\delta$-Galois group of \eqref{eq:rank1} is the
multiplicative group of constants $\Gm(\C)$.
\item $s_O\neq0$ for some moving orbit $O$. Then $\Lc(\hat\alpha) = 0$, the $\delta$-Galois
group of \eqref{eq:rank1} is the multiplicative group $\Gm(k)$, and every solution $\hat x$
is $\delta$-transcendental over $K$.
\end{enumerate}
\end{theorem}

\begin{proof}
The three cases are mutually exclusive and exhaust all possibilities, because the
condition of case (1) implies that of case (2). In case (1), $\hat\alpha^m\in\B_\sigma(K)$ by
Lemma \ref{lem:cob}, and $m$ is minimal with this property, so the classical Galois group is
$\mu_m$ by Proposition \ref{prop:rank1}; the $\delta$-Galois group is Zariski dense in it
\cite[Prop.~4.38]{HSS}, hence equal to it, $\mu_m$ being finite. The value of $\Lc(\hat\alpha)$
is given by Lemma \ref{lm:dicho1}. In cases (2) and (3) no power of $\hat\alpha$ is a
$\sigma$-coboundary, by Lemma \ref{lem:cob}, so Theorem \ref{thm:main} applies: in case (2)
$\Lc(\hat\alpha)=k[\delta]$ by Lemma \ref{lm:dicho1}, so $L_0=1$ and
$G=\V(\ell y)=\Gm(\C)$; in case (3) $\Lc(\hat\alpha)=0$ by Lemma \ref{lm:dicho1} and
$G=\Gm(k)$.
\end{proof}

\subsection{Separability}

For genuine cellular automata the trichotomy becomes a dichotomy, and a criterion that can be read off the
symbol. Write $\hat\alpha=N/d$ with $N\in\C[t][z^{\pm1}]$ and $d\in\C[t]$ a common
denominator of the coefficients.

\begin{lemma}[structure of the roots]\label{lem:roots}
With the notation of \eqref{eq:factor}:
\begin{enumerate}[label=\textup{(\roman*)},leftmargin=2.4em,topsep=2pt,itemsep=2pt]
\item the roots with $e_i<0$ are roots of $d$, hence lie in $\C$ and are $\delta$-constant;
\item no $\Z$-orbit contains both a moving and a $\delta$-constant root;
\item consequently every moving orbit satisfies $s_O\geq1>0$.
\end{enumerate}
\end{lemma}
\begin{proof}
(i) After reducing the fraction, the poles in $t$ can only come from $d\in\C[t]$. (ii) If
$\rho-\rho'\in\Z$ then $\delta\rho=\delta\rho'$. (iii) By (i) and (ii) a moving orbit
consists only of roots of $N$, all with $e_i\geq1$.
\end{proof}

\begin{theorem}[separability criterion]\label{thm:sep}
Let $\hat\alpha\in\C(t)[z^{\pm1}]$ be the symbol of a linear cellular automaton over $\Z$ \eqref{eq:rank1}, and assume $\hat\alpha\notin\C(t)$, that is, the automaton really moves mass. The following
are equivalent:
\begin{enumerate}[label=\textup{(\arabic*)},leftmargin=2.6em,itemsep=2pt]
\item $\hat\alpha(z,t)=\gamma(z)\,\psi(t)$ with $\gamma\in\C[z^{\pm1}]$ and
$\psi\in\C(t)^\times$ \textup{(}the symbol is \emph{separable}\textup{)};
\item $\Lc(\hat\alpha)=k[\delta]$;
\item $\dGal=\Gm(\C)$, of $\delta$-dimension $0$ and order $1$;
\item $\hat x$ is $\delta$-algebraic over $K$;
\item the transformed fundamental solution $\beta(z,t)$ based at $t_0$ satisfies a linear
differential equation of \emph{first order} in $z$, $P\,\delta\beta=Q\,\beta$, with
$P,Q\in\C[t][z^{\pm1}]$, $P\neq0$, valid for every $t\geq t_0$.
\end{enumerate}
Here $\hat x$ denotes a solution of \eqref{eq:rank1} in a parameterized Picard--Vessiot ring
$R$. When the conditions hold, $\ell(\hat x)=t\,\ell(\gamma)+v$ with $v\in k$, while
$\ell(\beta)=(t-t_0)\,\ell(\gamma)$. Otherwise $\dGal=\Gm$ and $\hat x$ is
$\delta$-transcendental over $K$.
\end{theorem}

\begin{proof}
(1)$\Rightarrow$(2). If $\hat\alpha=\gamma\psi$ all roots $\rho_i$ are roots of
$\psi\in\C(t)$, hence lie in $\C$ and there are no moving orbits; apply
Lemma \ref{lm:dicho1}.

(2)$\Rightarrow$(1). Suppose that there are no moving orbits, so that every $\rho_i$ is a
$\delta$-constant of $k$, that is $\rho_i\in\C$. Then $\psi:=\prod_i(t-\rho_i)^{e_i}\in\C(t)^\times$
and $\gamma=\hat\alpha/\psi$ is an element of $k$ lying in $\C(z,t)$, hence in $\C(z)$; since
$\gamma\psi=\hat\alpha\in\C(t)[z^{\pm1}]$, we get
$\gamma\in\C(t)[z^{\pm1}]\cap\C(z)=\C[z^{\pm1}]$, and $\hat\alpha$ is separable. Hence, if
$\hat\alpha$ is not separable, there is a moving orbit; by Lemma \ref{lem:roots}(iii) it has
$s_O\neq0$, and Lemma \ref{lm:dicho1} gives $\Lc=0$.

(2)$\Leftrightarrow$(3) and (3)$\Rightarrow$(4) are Theorem \ref{thm:main}, and
(4)$\Rightarrow$(2) follows from Theorem \ref{thm:main} together with
Lemma \ref{lm:dicho1}, which excludes $0\subsetneq\Lc\subsetneq k[\delta]$.

(1)$\Rightarrow$(5). With $\hat\alpha=\gamma\psi$ one has $b=\ell(\gamma)\in\C(z)$, which
is $\sigma$-constant; since $\ell$ takes products to sums,
$\ell(\beta)=\sum_{s=t_0}^{t-1}b=(t-t_0)\,\ell(\gamma)$, a quotient of
elements of $\C[t][z^{\pm1}]$. Writing it as $Q/P$ gives $P\,\delta\beta=Q\,\beta$ for every
$t\geq t_0$. For a solution $\hat x\in R$, since $b=\ell(\gamma)\in k$ is $\sigma$-constant,
$b=(\sigma-1)(t\,b)$, and Theorem \ref{thm:main}(2) with $L_0=1$ and $g=t\,\ell(\gamma)$
gives $\ell(\hat x)-t\,\ell(\gamma)\in k$, which is the relation stated after the list.

(5)$\Rightarrow$(4). The identity $P\,\delta\beta=Q\,\beta$, valid for all $t\geq t_0$, is
a spatial recurrence of the type considered in \S\ref{sec:holo}, and Theorem \ref{thm:holoA}
(whose proof uses only Lemmas \ref{lem:riccati}--\ref{lem:embed}, and not the present
theorem) shows that $\hat x$ is then $\delta$-algebraic over $K$.

Finally no power of $\hat\alpha$ is a $\sigma$-coboundary. Indeed, by Lemma \ref{lem:cob},
$\hat\alpha^m\in\B_\sigma(K)$ would require $s_O=0$ for \emph{every} orbit and $\gamma$ to
be a root of unity. If some root is moving, the first condition fails by
Lemma \ref{lem:roots}(iii). If none is, then $\hat\alpha=\gamma\psi$ with
$\gamma\in\C[z^{\pm1}]$ the leading coefficient of $N$ in $t$; a root of unity in
$\C[z^{\pm1}]$ must be a constant, since the units of $\C[z^{\pm1}]$ are the monomials and
$(cz^a)^m=1$ forces $a=0$; then $\hat\alpha\in\C(t)$, which was excluded.
\end{proof}

\section{Computation of the $\delta$-Galois group over $K=\C(z)^{\mathrm{diff}}(t)$}
\label{sec:algo}

\subsection{The algorithm}

Theorems \ref{thm:tricho} and \ref{thm:sep} turn the computation of the $\delta$-Galois
group of \eqref{eq:rank1} into a finite procedure. If the symbol $\hat\alpha \in \Q(t)[z^{\pm1}]$ has rational
coefficients, all steps are effective: factorization over $\overline{\Q(z)}$ is standard,
and the test $\rho-\rho'\in\Z$ is decided by the dispersion computations used in symbolic
summation \cite{Abramov}. In fact no algebraic closure is needed: it suffices to factor
$\mathrm{num}$ and $\mathrm{den}$ into monic irreducible factors over $\Q(z)$; all roots of
such a factor are conjugate over $\Q(z)$, and they are $\delta$-constant if and only if the
factor has coefficients in $\Q$, since the field of constants of $\overline{\Q(z)}$ is
$\overline\Q$.

\begin{algorithm}[h]
\caption{$\delta$-Galois group of a linear cellular automaton over $\Z$}
\label{alg:main}
\begin{algorithmic}[1]
\Require $\hat\alpha\in\Q(z)(t)^\times$, the symbol.
\Ensure the group $\dGal$, and the relation \eqref{eq:relation} when it is not $\Gm$.
\State write $\hat\alpha=\mathrm{num}/\mathrm{den}$ in lowest terms and factor both as
  polynomials in $t$ over $\overline{\Q(z)}$, obtaining
  $\hat\alpha=\gamma\prod_i(t-\rho_i)^{e_i}$ with $\gamma\in\Q(z)^\times$
\State group the $\rho_i$ into $\Z$-orbits: $\rho\sim\rho'$ iff $\rho-\rho'\in\Z$
\State set aside the orbits with $\partial_z\rho=0$ \Comment{$\delta$-constant, no contribution to $\Lc$}
\ForAll{moving orbits $O$}
  \State $s_O\gets\sum_{\rho_i\in O}e_i$
  \If{$s_O\neq0$} \Return $\dGal=\Gm$; the solution is $\delta$-transcendental \EndIf
\EndFor
\State \Comment{all moving orbital sums vanish}
\If{$s_O=0$ for \emph{every} orbit, moving or not, \textbf{and} $\gamma\in\{\pm1\}$}
  \State \Return the finite group $\mu_m$, $m=\ord(\gamma)$
  \Comment{over $\Q(z)$ the only roots of unity are $\pm1$}
\EndIf
\State compute $b\gets z\,\partial_z\hat\alpha/\hat\alpha$ and a certificate $g$ with
  $\sigma(g)-g=b$ by Abramov's algorithm
\State \Return $\dGal=\Gm(\C)$, with relation $\ell(\hat x)=g+v$, $v\in k$
\end{algorithmic}
\end{algorithm}

\noindent The test of lines 10--12 is the verification of the first case of the trichotomy in Theorem
\ref{thm:tricho}: without it the algorithm would return $\Gm(\C)$ for
$\hat\alpha=(t+1)/t=\sigma(t)/t$, whose Galois group is trivial, and for
$\hat\alpha=-(t+1)/t$, whose group is $\mu_2$. For \emph{genuine} cellular automata
($\hat\alpha\notin\C(t)$) that step is superfluous, its implicit verification in
the proof of Theorem \ref{thm:sep} short-circuits the whole
algorithm. It suffices to ask whether $\hat\alpha$ factors as a function of $z$ times a
function of $t$. Steps 1--8 are then a way of certifying the answer.

\subsection{From the group to a recurrence}

In the separable case the algorithm returns more than the group. Writing
$(t-t_0)\,\ell(\gamma)=Q/P$ with $P,Q\in\C[t][z^{\pm1}]$, $P=\sum_jP_j(t)z^j$,
$Q=\sum_jQ_j(t)z^j$, the identity $P\,\delta\beta=Q\,\beta$ satisfied by the transformed
fundamental solution $\beta(z,t)=\sum_nx_n(t)z^n$ and $\delta z^n=nz^n$ give,
upon comparing coefficients of $z^n$, the linear recurrence
\begin{equation}\label{eq:rec}
\sum_{j}\Big[(n-j)\,P_j(t)-Q_j(t)\Big]\,x_{n-j}(t)=0 ,
\end{equation}
valid for all $n\in\Z$ and all $t\geq t_0$. We return to this in \S\ref{sec:holo}.

\subsection{Wolfram's rule 90 and the binomial coefficients}

\begin{example}[Pascal]\label{ex:pascal}
Consider the autonomous automaton
\[
s_{n}(t+1)=s_{n-1}(t)+s_{n+1}(t),\qquad n\in\Z,
\]
which is Wolfram's rule 90 with complex states \cite[Ex.~6]{BM}. Its symbol is
$\hat\alpha=z+z^{-1}$, trivially separable with $\psi=1$, so by Theorem \ref{thm:sep}
\[
\dGal=\Gm(\C),\qquad \ddim=0,\qquad\ord=1 .
\]
The certificate is $b=\ell(\hat\alpha)=\dfrac{z^2-1}{z^2+1}$ and, with $t_0=0$, the
transformed fundamental solution $\beta=\hat\alpha^t$ satisfies
\[
\ell(\beta)=t\,\frac{z^2-1}{z^2+1},
\]
which under $z=e^{i\theta}$ reads $\partial_\theta\beta/\beta=-t\tan\theta$. Since
$\hat\alpha^t=(2\cos\theta)^t$, the relation says exactly that $(2\cos\theta)^t$ satisfies
a first order linear differential equation in the angle, with coefficient linear in $t$.
The finite support solutions are the binomial coefficients: the fundamental solution is
\[
f_{n}(t)=\binom{t}{\tfrac{t+n}{2}}\quad\text{if }t+n\text{ is even},\qquad 0
\text{ otherwise},
\]
and \eqref{eq:rec} becomes the recurrence $(n+t)\,x_n=(t+2-n)\,x_{n-2}$, which we have
verified against the automaton for $t\leq12$.
\end{example}

\subsection{Stirling numbers and H\"older's theorem}

\begin{example}[Stirling]\label{ex:stirling}
Consider now the non-autonomous automaton
\[
s_{n}(t+1)=s_{n-1}(t)+t\,s_{n}(t),\qquad n\in\Z,\qquad s_{n}(0)=\delta_{n,0},
\]
with $\kappa=\C(t)$. Its symbol is $\hat\alpha=z+t$, which is \emph{not} separable. Then:
\begin{enumerate}[label=\textup{(\roman*)},leftmargin=2.4em,itemsep=2pt]
\item the fundamental solution is $f_{n}(t)=\genfrac[]{0pt}{}{t}{n}$, the unsigned
Stirling number of the first kind, since the triangular recurrence
$\genfrac[]{0pt}{}{t+1}{n}=\genfrac[]{0pt}{}{t}{n-1}+t\genfrac[]{0pt}{}{t}{n}$ is
literally the rule of the automaton;
\item its Fourier transform is the Pochhammer symbol
\[
\beta(z,t)=\prod_{j=0}^{t-1}(z+j)=(z)_t=\frac{\Gamma(z+t)}{\Gamma(z)},
\]
by the classical generating identity for Stirling numbers;
\item the unique root in $t$ is $\rho=-z$, which is moving, with $e=1$ and $s_O=1\neq0$;
hence by Theorem \ref{thm:tricho}
\[
\dGal=\Gm,\qquad \ddim=1,\qquad \dtrdeg\big(Q(R)\mid K\big)=1 .
\]
\end{enumerate}
In particular a solution $\hat x\in R$ of $\sigma(\hat x)=(z+t)\hat x$ satisfies no
algebraic differential equation in $z$ with coefficients in $K$. By the realization argument
of Theorem \ref{thm:holoA} (Lemma \ref{lem:embed} applied in the ring of germs
$\Seq_\infty(k)$), the same holds for the sequence of polynomials
$\big((z)_t\big)_{t\in\N}=\big(\Gamma(z+t)/\Gamma(z)\big)_{t\in\N}$: no nonzero differential
polynomial with coefficients in $K$ vanishes on it for all large $t$. This is a cellular
automaton analogue of H\"older's theorem \cite{Holder}; the proof follows the method of Hardouin and Singer \cite{HS}, of which
Theorem \ref{thm:main} is a specialization. We say analogue and not corollary: H\"older's
theorem concerns $\Gamma$ as a function of one variable, whereas the statement above
concerns $\Gamma(z+t)/\Gamma(z)$ over the field $K=k(t)$, and neither is formally an
immediate consequence of the other. What is new here is not the technique but the source:
the symbol $z+t$ is the simplest non-autonomous rule on $\Z$, so the analogue of $\Gamma$
appears as a two term cellular automaton.
\end{example}

\section{$\delta$-Galois versus holonomy}\label{sec:holo}

The relation in the angle variable and the recurrence in the space variable are the same
object read on the two sides of the Fourier transform. We make this precise; the outcome is
that the parameterized group in the time direction decides holonomy in the space direction.

\subsection{Spatial holonomy}

Let $\beta(z,t)=\prod_{s=t_0}^{t-1}\hat\alpha(z,s)=\sum_nx_n(t)z^n\in\C[z^{\pm1}]$ be the
transformed fundamental solution based at $t_0$ as in \eqref{eq:t0}, so that $x_n(t)$ are
the coefficients of the fundamental solution. By the
choice of $t_0$ it is a nonzero Laurent polynomial for every $t\geq t_0$, hence invertible
in $\C(z)$.

\begin{definition}\label{def:holo}
The automaton is \emph{spatially holonomic} if there exist $d\geq0$ and
$p_0,\dots,p_d\in\C[n,t]$, not all zero, with
\begin{equation}\label{eq:U}
\sum_{j=0}^{d}p_j(n,t)\,x_{n+j}(t)=0\qquad\text{for all }n\in\Z\text{ and all }t\geq t_0.
\tag{U}
\end{equation}
Requiring \eqref{eq:U} for all $t\geq t_0$ or only for all large $t$, as an identity of germs
in the sense of \S\ref{sec:auto}, gives the same notion: if \eqref{eq:U} holds for $t\geq T$,
multiplying every $p_j$ by $\prod_{s=t_0}^{T-1}(t-s)$ makes it hold for $t\geq t_0$. The notion
does not depend on the choice of $t_0$ either: enlarging $t_0$ divides $\beta$ by a fixed
Laurent polynomial $D(z)$, and if $\Theta\beta=0$ then $\Theta\circ D$ annihilates $\beta/D$,
while if $\Theta(\beta/D)=0$ then clearing denominators in $\Theta\circ D^{-1}$ gives an
operator of $\C[t][z^{\pm1}][\delta]$ annihilating $\beta$; Lemma \ref{lem:dict} converts
operators back into recurrences.
\end{definition}

The uniformity in $t$ carries all the weight. For each fixed $t$ the family $(x_n(t))_n$ has finite
support, hence satisfies recurrences trivially: taking $d=0$ and
$p_0(n)=\prod_{m=N_-}^{N_+}(n-m)$, where $[N_-,N_+]$ contains the support, already works.
It also shows why $t_0$ matters: if the symbol were allowed to vanish at some $s\in\N$ and
we based the solution at $0$, then $\beta(z,t)$ would be identically zero for $t>s$ and
\eqref{eq:U} would hold for trivial reasons. Basing at $t_0$ excludes this.
What \eqref{eq:U} demands is \emph{one} identity whose coefficients are polynomial in $t$,
so that the same recurrence serves all times.

\subsection{The operator dictionary}

Let $F$ be a field of characteristic zero and let $\Dc_F$ be the algebra of operators on
$F[z^{\pm1}]$ generated by the multiplications $z^a$, the scalars of $F$ and
$\delta=z\,d/dz$. From $\delta z^a=z^a(\delta+a)$, every element of $\Dc_F$ is uniquely
$\Theta=\sum_{a,b}c_{a,b}z^{a}\delta^{b}$.

\begin{lemma}[dictionary]\label{lem:dict}
Let $f=\sum_nu_nz^n$. Then
\[
\Big(\sum_{a,b}c_{a,b}z^a\delta^b\Big)f=\sum_n\Big(\sum_aP_a(n-a)u_{n-a}\Big)z^n,
\qquad P_a(n):=\sum_bc_{a,b}n^{b},
\]
so $\Theta f=0$ if and only if $(u_n)$ satisfies the recurrence
$\sum_aP_a(n-a)u_{n-a}=0$; and $\Theta\leftrightarrow(P_a)_a$ is a bijection between
operators and linear recurrences with polynomial coefficients.
\end{lemma}
\begin{proof}
$\delta^bz^m=m^bz^m$, so $z^a\delta^bf=\sum_n(n-a)^bu_{n-a}z^n$.
\end{proof}

\noindent This elementary computation is the classical equivalence between $D$-finiteness
and P-recursivity \cite{Stanley,Lipshitz}; stated with the Euler operator it makes
transparent that the two notions are the same one, written on the two sides of the Fourier
transform.

\subsection{From holonomy to the group}

\begin{lemma}[Riccati sequence]\label{lem:riccati}
Let $y$ be invertible and $u=\ell(y)$. Define $R_0=1$, $R_{i+1}=\delta R_i+uR_i$. Then
$\delta^iy=R_iy$, and
\[
R_i=u^{i}+(\text{terms containing at least one }\delta^{j}u,\ j\geq1).
\]
\end{lemma}
\begin{proof}
The first assertion is induction. For the second, $R_0=1$ and $R_1=u$; if $R_i=u^i+S_i$
with $S_i$ made of terms containing derivatives, then every term of $\delta R_i$ contains a
derivative by the Leibniz rule, and $uR_i=u^{i+1}+uS_i$.
\end{proof}

\begin{lemma}\label{lem:uy}
Let $K\subseteq L$ be $\delta$-fields and $y\in L^\times$ with $u=\ell(y)$. Then $y$ is
$\delta$-algebraic over $K$ if and only if $u$ is.
\end{lemma}
\begin{proof}
One direction is clear. Conversely $\delta y=uy$ is an algebraic differential equation of
first order for $y$ over $K\langle u\rangle$, so
$\dtrdeg(K\langle y\rangle\mid K\langle u\rangle)=0$; conclude by additivity of the
differential transcendence degree in towers.
\end{proof}

\begin{lemma}[realization]\label{lem:embed}
Let $S$ be a $\sigma\delta$-$K$-algebra containing an invertible solution $Y$ of
$\sigma(X)=\hat\alpha X$, and let $S'\subseteq S$ be the $\sigma\delta$-$K$-subalgebra
generated by $Y$ and $Y^{-1}$. Then there is a maximal $\sigma\delta$-ideal
$\mathfrak m\subset S'$, an isomorphism of $\sigma\delta$-$K$-algebras
$S'/\mathfrak m\cong R$ with $R$ a parameterized Picard--Vessiot ring for the equation,
and under it the class of $Y$ is $c\,\hat x$ for some $c\in k^\times$.
\end{lemma}
\begin{proof}
$S'\neq0$, so Zorn's lemma provides a maximal $\sigma\delta$-ideal $\mathfrak m$; the
quotient $S'/\mathfrak m$ is $\sigma\delta$-simple and is generated over $K$ by the class
of $Y$ together with its inverse, that is, by a fundamental solution and the inverse of its
determinant. Since $k=K^\sigma$ is $\delta$-closed, a $\sigma\delta$-simple, finitely
$\delta$-generated $\sigma\delta$-$K$-algebra has ring of $\sigma$-constants equal to $k$
\cite{HS,Wibmer}, so $S'/\mathfrak m$ is a parameterized Picard--Vessiot ring; by
uniqueness it is isomorphic to $R$. Two invertible solutions in a $\sigma$-ring with
constants $k$ differ by a factor in $k^\times$, which gives the last assertion.
\end{proof}

\noindent We use the lemma in the opposite direction to an embedding: it lets us transport
an identity valid in $S$ down to $R$, which is all that is needed below.

\begin{theorem}\label{thm:holoA}
Assume that $\hat\alpha^m\notin\B_\sigma(K)$ for all $m\geq1$, which is automatic for
genuine automata by the last paragraph of the proof of Theorem \ref{thm:sep}. If the automaton
is spatially holonomic then
$\Lc(\hat\alpha)\neq0$, equivalently the parameterized Galois group has
$\delta$-dimension $0$.
\end{theorem}
\begin{proof}
By Lemma \ref{lem:dict}, \eqref{eq:U} corresponds to a nonzero operator
$\Theta=\sum_{b\leq r}q_b(z,t)\delta^{b}$ with $q_b\in\C[t][z^{\pm1}]$, $q_r\neq0$, such
that $\Theta(\beta(z,t))=0$ for every $t\geq t_0$. Let $u(t)=\ell(\beta(z,t))\in\C(z)$.
Dividing by $\beta(z,t)$ and using Lemma \ref{lem:riccati},
\begin{equation}\label{eq:ADE}
\Phi\big(u,\delta u,\dots,\delta^{r-1}u\big):=\sum_{b=0}^{r}q_b(z,t)\,R_b=0 .
\end{equation}
By Lemma \ref{lem:riccati} each $R_b$ has total degree $b$ in $u,\delta u,\dots$, and the
pure power $u^{r}$ occurs only in $R_r$; so the coefficient of $u^{r}$ in $\Phi$ is
$q_r\neq0$ and $\Phi$ is a nonzero
differential polynomial with coefficients in $\C[t][z^{\pm1}]\subseteq K$.

It remains to see that \eqref{eq:ADE} holds in the parameterized Picard--Vessiot extension.
Take $S:=\Seq_\infty(k)$, the ring of germs of sequences in $k$, with $\sigma$ the shift and
$\delta$ acting termwise; then $\sigma$ is an automorphism, $S^\sigma=k$, and
$K=k(t)$ embeds in $S$ by sending $t$ to the germ of $(0,1,2,\dots)$ --- here again germs
dispose of the finitely many integers at which a rational function is undefined. The germ
$Y$ of $t\mapsto\beta(z,t)$ is an invertible solution of $\sigma(X)=\hat\alpha X$ in $S$,
invertibility being exactly what the choice of $t_0$ in \eqref{eq:t0} guarantees. Identity
\eqref{eq:ADE} holds for every $t\geq t_0$, hence in $S$, hence in the
$\sigma\delta$-$K$-subalgebra $S'$ generated by $Y^{\pm1}$, and therefore in every quotient
of $S'$. By Lemma \ref{lem:embed} one such quotient is a parameterized Picard--Vessiot ring
$R$, in which the class of $Y$ equals $c\,\hat x$ with $c\in k^\times$; so $\Phi$
annihilates $\ell(\hat x)+\ell(c)$ there. As $\ell(c)\in k\subseteq K$, replacing $u$ by
$u+\ell(c)$ turns $\Phi$ into another nonzero differential polynomial over $K$ with the
same term $q_ru^{r}$ of top degree. Hence $\ell(\hat x)$ is $\delta$-algebraic over $K$. By Lemma \ref{lem:uy} so is $\hat x$, whence
$\dtrdeg(Q(R)\mid K)=0$ and $\Lc(\hat\alpha)\neq0$ by Theorem \ref{thm:main}.
\end{proof}

\subsection{The equivalence}

\begin{theorem}\label{thm:holoequiv}
Let $\hat\alpha\in\C(t)[z^{\pm1}]$ with $\hat\alpha\notin\C(t)$. Then the conditions
\textup{(1)--(5)} of Theorem \ref{thm:sep} are further equivalent to
\begin{enumerate}[label=\textup{(\arabic*)},leftmargin=2.6em,start=6]
\item the automaton is spatially holonomic.
\end{enumerate}
When they hold, the recurrence is \eqref{eq:rec} with $Q/P=(t-t_0)\,\ell(\gamma)$, of order equal
to the length of the joint support of $P$ and $Q$ and of degree one in $n$. Otherwise
$\dGal=\Gm$ and \emph{no} uniform recurrence exists.
\end{theorem}
\begin{proof}
(5)$\Rightarrow$(6) is the computation \eqref{eq:rec}, which is Lemma \ref{lem:dict}
applied to $\Theta=P\delta-Q$. (6)$\Rightarrow$(2) is Theorem \ref{thm:holoA}, whose
hypothesis holds for genuine automata, together with Lemma \ref{lm:dicho1}, which excludes
$0\subsetneq\Lc\subsetneq k[\delta]$.
\end{proof}

\begin{corollary}\label{cor:stirling}
There is no nonzero family $p_0,\dots,p_d\in\C[n,t]$ such that
\[
\sum_{j=0}^{d}p_j(n,t)\,\genfrac[]{0pt}{}{t}{n+j}=0
\qquad\text{for all }n\in\Z\text{ and all }t\in\N .
\]
Equivalently, no nonzero operator of $\C[t][z^{\pm1}][\delta]$ annihilates $(z)_t$ for
every $t\in\N$. In particular the bivariate array
$\big(\genfrac[]{0pt}{}{t}{n}\big)_{n,t}$ is not P-finite in the sense of \cite{Lipshitz}.
\end{corollary}
\begin{proof}
The symbol $z+t$ is not separable and $t_0=0$ works for it, so Theorem
\ref{thm:holoequiv} applies. The equivalence with the operator statement is
Lemma \ref{lem:dict}. For the last assertion, P-finiteness requires a recurrence in $n$
alone whose coefficients are polynomial in $n$ and $t$, but only for $n$ large; this
suffices, because the support is bounded below uniformly in $t$ --- one has
$\genfrac[]{0pt}{}{t}{n}=0$ for $n\leq0<t$ --- so that a recurrence valid for $n\geq n_0$
becomes valid for all $n\in\Z$ after multiplying the coefficients by
$\prod_{i=-d}^{n_0-1}(n-i)$ (the factor $n+d$ takes care of $t=0$, where
$\genfrac[]{0pt}{}{0}{0}=1$).
\end{proof}

\begin{remark}
For each fixed $t\in\N$ the function $(z)_t$ is a
polynomial, hence $D$-finite in $z$; and for a fixed complex parameter $t\notin\Z$ the
function $\Gamma(z+t)/\Gamma(z)$ has infinitely many poles and is trivially not $D$-finite.
The content of the corollary is the uniformity in $t$, which is precisely what the
parameterized Galois group measures.
\end{remark}

\begin{remark}\label{rmk:status}
The non-holonomicity of the Stirling array is treated as known in the symbolic summation
literature: Kauers opens \cite{Kauers07} by observing that no summation algorithm covers
Stirling number identities, because their triangular recurrences exclude them from all
tractable classes, and Chyzak, Kauers and Salvy build the framework of \cite{CKS09}
precisely because Stirling numbers are not covered by the holonomic one, asserting that the
array is not $\partial$-finite. We have not located a published proof. The papers that do
prove non-holonomicity --- \cite{Gerhold04} for $\sqrt n$, $n^n$, $W(n)$; \cite{FGS,FGS10}
for $\log n$, $p_n$, $1/H_n$, $e^{cn^{\theta}}$; \cite{BGKL08} for values of $\zeta$
at integer arguments --- do not include the Stirling numbers, and their methods are
mostly analytic.
The closest methodological precedent is \cite{BDR}, where the differential transcendence of
the Bell numbers, first proved by Klazar \cite{Klazar}, is reproved by difference Galois
theory. For the Stirling array itself a short direct argument is available
(Remark \ref{rmk:egf}); what Theorem \ref{thm:holoequiv} adds is a criterion valid for the
whole family of automata over $\Z$, rather than an argument case by case.
\end{remark}

\begin{remark}[a direct argument for the Stirling array]\label{rmk:egf}
Corollary \ref{cor:stirling} can also be obtained from the exponential generating function
\[
F(z,u)=\sum_{t\geq0}(z)_t\,\frac{u^t}{t!}=(1-u)^{-z}=e^{-zL},\qquad L=\log(1-u),
\]
regarded as a formal power series in $u$ with coefficients in $\C[z]$. Let
$\Theta=\sum_{b\leq r}q_b(z,t)\delta^b$, $q_r\neq0$, annihilate $(z)_t$ for every $t\in\N$.
Since multiplication by $t$ on coefficients corresponds to $\theta_u=u\,\partial_u$ on
exponential generating functions, and $t$ commutes with $z$ and $\delta=\delta_z$, we get
$\Theta(z,\delta_z,\theta_u)F=0$. Now $\delta_zF=-zL\,F$ and $\theta_uF=\frac{zu}{1-u}F$,
while $\theta_u L$ is rational; hence $\theta_u$ does not raise the degree in $L$, and
$\Theta F=F\cdot P(L)$ with $P\in\C(z,u)[L]$. The terms with $b<r$ contribute in degree at
most $b$, and the coefficient of $L^r$ is $(-z)^r\,q_r(z,\theta_u)F/F$, which is nonzero
because $q_r(z,\theta_u)F=\sum_tq_r(z,t)(z)_t\,u^t/t!\neq0$. Thus $L$ would be algebraic
over $\C(z,u)$, which is absurd. This argument depends on the closed form of $F$ and does
not extend to general symbols, whereas Theorem \ref{thm:holoequiv} does.
\end{remark}


\section{Conclusions}\label{sec:concl}

For linear cellular automata over $\Z$ we have computed the parameterized Galois group of
the Fourier transformed equation, found that for genuine automata with symbols rational in
time only two groups occur, and
identified the dividing line as separability of the symbol, a condition that can be checked
by inspection and certified by an algorithm. The two simplest automata on $\Z$ realize the
two cases, and the non-autonomous one is a cellular automaton form of H\"older's theorem.

Placing this beside the results of \cite{BM} we obtain a coherent picture of the three
groups attached to an automaton. They are the three natural ways of looking at the group
algebra $\A=\C[z^{\pm1}]$: its torsion quotients, its fraction field, and $\A$ itself. The
periodic group is the character group of the multiplicative group generated by the values
of the symbol at roots of unity, an arithmetic object; the parameterized group lives at the
generic point, where the classical theory is blind and all the content is differential; and
the finite support group sits in between and remains open.
Finally, the parameterized group in the time direction decides holonomy in the space
direction, which gives a Galois theoretic route to non-holonomicity results.

\section*{Declaration of AI-assisted technologies}

During the preparation of this work the authors used Claude Opus versions 5 and  5.5 (Anthropic), a large language
model, over an extended series of working sessions in dialogue with the first author. Starting with a draft prepared by the authors, it was used in four ways:
\emph{(i) auditing}: reading successive drafts and reporting errors, gaps and inconsistencies
in statements and proofs, each of which was then checked by the authors;
\emph{(ii) drafting and reformulation}: proposing wording, structure and alternative proofs
for parts of the text, which the authors rewrote, accepted or discarded;
\emph{(iii) verification}: checking computations such as those in Examples \ref{ex:twobases}
and \ref{ex:pascal} by symbolic computation in SymPy;
\emph{(iv) literature search}: locating and summarizing references, all of which were checked
against the original sources by the authors.

\medskip
The problem, the framework and the antecedent
results are the authors' \cite{BM}. Throughout the dialogue the  authors directed the
line of inquiry, posed the questions, and repeatedly corrected
the tool. We record, as a matter of intellectual honesty
about this mode of work, that the tool produced a substantial number of incorrect
statements which were subsequently detected. The practical lesson we draw is that a draft
produced this way requires an audit at least as careful as a referee's, and that asking the
tool to audit its own output is productive but not sufficient.

\medskip

The authors have reviewed, corrected, and verified the mathematical statements in this text and claim full responsibility for all mathematical content.

\section*{Acknowledgements}

We acknowledge the support of our host institution Universidad Nacional de Colombia -- Sede Medell\'in. DBS acknowledges the support of Facultad de Ciencias' project ``Algunas conexiones de la teoría de Picard-Vessiot con la geometría y la din\'amica'' (hermes 67312).


\end{document}